\documentclass[reqno]{amsart}
\usepackage{amsmath}
\usepackage{amssymb}
\usepackage{amsfonts}
\usepackage{amsthm}
\def\negskp{\unskip\kern -2pt\unskip}

\newtheorem{theorem}{Theorem}

\newtheorem{lemma}{Lemma}

\numberwithin{conjecture}{section}
\numberwithin{theorem}{section}
\numberwithin{lemma}{section}
\numberwithin{corollary}{section}
\numberwithin{proposition}{section}
\numberwithin{equation}{section}
\begin{document}

\thispagestyle{empty}
\title[A general rational solution \dots]
{A general rational solution of an equation associated with perfect cuboids.}
\author{John R. Ramsden}
\address{CSR Plc,
\newline\hphantom{iii}\, Cambridge Business Park,
\newline\hphantom{iii}\, CAMBRIDGE,
\newline\hphantom{iii}\, CB4 0WZ, UK}
\email{John.Ramsden@csr.com
\newline\hphantom{iii}\kern 7.2em jhnrmsdn@yahoo.co.uk
}
\thanks{\copyright \ 2012 John R Ramsden}
\thanks{\it 2012-07-22}
\maketitle 
{
\small
\begin{quote}
\noindent{\bf 
We show by finding an explicit parametrization that a 4th degree surface
which arises as a necessary condition for the existence of a perfect cuboid
is a rational surface, i.\,e. birationally equivalent over $\mathbb{Q}$ to
a plane.}\par
\end{quote}
}
\section{Introduction}\label{sec:intro}
A recent paper, \cite{1}, contains, among other interesting results, a
necessary condition for the existence of perfect cuboids. This condition
comprises rational solutions to an equation (labeled (7.3) in this paper)
which, with a slight change of notation, can be represented as follows:
\begin{equation}\label{eq:1}
\hskip -2em
\begin{gathered}
4\,x^2 + (y^2 + 1 - z^2)^2 = 8\,y^2.
\end{gathered}
\end{equation}
The purpose of the present note is to show that this surface is birationally
equivalent over $\mathbb{Q}$ to a plane. We show this by deriving an explicit
parametrization, which defines every point of the surface except those on 8
lines it contains defined by $x = \pm y = \pm z \pm 1$ for 3 independent
sign options.
\section{Results}\label{sec:results}
We start with a couple of lemmas, which form a part of a ``toolkit'' for finding
parametrizations, where these exist, by elementary techniques of the kind used
in the present note.
\begin{lemma}\label{lemma1}
Let $x$, $y$, $z$ be rational numbers with $1 + x^2 = y^2 + z^2$ and $x \not= z$.
Then there exist rational numbers $a$, $b$ such that
\begin{xalignat*}{3}
&x=\frac{a\,b + 1}{a - b},
&&y=\frac{a + b}{a - b},
&&z=\frac{a\,b - 1}{a - b}.
\end{xalignat*}
\end{lemma}
\begin{proof} The numbers $a$ and $b$ defined as follows must be rational:
\begin{xalignat*}{2}
&a=\frac{y + 1}{x - z},
&&b=\frac{x + z}{y + 1}.
\end{xalignat*}
Then, expressing the equation of the lemma in the form
\begin{equation*}
(x + z)(x - z) = (y + 1)(y - 1),
\end{equation*}
the preceding equations defining $a$ and $b$ imply
\begin{equation*}
y = \frac{a + b}{a - b},
\end{equation*}
whence
\begin{align*}
  &x-z=\frac{y + 1}{a}=\frac{2}{a - b},\\[1ex]
  &x+z=b\,(y + 1)=\frac{2\,a\,b}{a - b}
\end{align*}
and the result follows.
\end{proof}
\begin{lemma}\label{lemma2} Let $x$, $y$ be rational numbers such that
$x^2 - y^2 = a$ for some non-zero $a$. Then there exists some rational
number $t$ for which
\begin{xalignat*}{2}
&x=\frac{t^2+a}{2\,t},
&&y=\frac{t^2-a}{2\,t}.
\end{xalignat*}
\end{lemma}
\begin{proof}
The proof is immediate, by adding and subtracting the following, which
must clearly hold for some non-zero rational $t$:
\begin{xalignat*}{2}
&x+y=t,
&&x-y=\frac{a}{t}.
\end{xalignat*}
\end{proof}
\par
We now state and prove our main theorem.
\begin{theorem}\label{theof2.1}
Every rational point on the surface defined by \eqref{eq:1} is either on one
of the 8 lines defined by $x = \pm y = \pm z \pm 1$ with 3 independent sign
options, or else there exists a pair of rational numbers $b$, $c$ with
\begin{align*}
&x=- \frac{b\,(c^2+2-4\,c)}{b^2\,c^2+2\,b^2-3\,b^2\,c+c-b\,c^2\,+2\,b},\\[1ex]
&y=- \frac{b(c^2+2-2\,c)}{b^2\,c^2+2\,b^2-3\,b^2\,c+c-b\,c^2+2\,b},\\[1ex]
&z=- \frac{b^2\,c^2+2\,b^2-3\,b^2\,c\,-c}{b^2\,c^2+2\,b^2-3\,b^2\,c+c-b\,c^2+2\,b}.
\end{align*}
Conversely this set satisfies \eqref{eq:1} identically, and we have
\begin{xalignat*}{2}
&b=\frac{y - x}{u - z - 1},\\[1ex]
&c=\frac{4\,x\,(y-x)}{(x-u)(y+x)+(u+3\,x)(y-x)},
\end{xalignat*}
where $u$ is defined by
\begin{equation}\label{eq:2}
\hskip -2em
y^2+1-z^2=2\,u.
\end{equation}
\end{theorem}
\begin{proof}
Substituting $u$ defined by \eqref{eq:2} into \eqref{eq:1} gives
\begin{equation}\label{eq:3}
\hskip -2em
\begin{gathered}
  x^2 + u^2 = 2\,y^2
\end{gathered}
\end{equation}
and adding \eqref{eq:2} to this gives
\begin{equation}\label{eq:4}
\hskip -2em
\begin{gathered}
  x^2 + (u - 1)^2 = y^2 + z^2.
\end{gathered}
\end{equation}
In \eqref{eq:3} and \eqref{eq:4}, dividing throughout by $x^2$ and taking
\begin{align*}
&U=\frac{u}{x}, \\[1ex]
&X=U-\frac{1}{x}, \\[1ex]
&Y=\frac{y}{x}, \\[1ex]
&Z=\frac{z}{x},
\end{align*}
this pair becomes
\begin{align}
&1+U^2=2\,Y^2,     \label{eq:6-1} \\
&1+X^2=Y^2+Z^2.    \label{eq:6-2}
\end{align}
By Lemma~\ref{lemma1}, \eqref{eq:6-2} implies either $X = Z$ or the
existence of rational $a$ and $b$ with
\begin{align}
&X=\frac{a\,b+1}{a-b}, \notag \\
&Y=\frac{a+b}{a-b},   \label{eq:7-2} \\
&Z=\frac{a\,b-1}{a-b}.\notag
\end{align}
The equality $X=Z$ leads to the 8 lines. Otherwise, plugging into \eqref{eq:6-1}
the expression for $Y$ given by \eqref{eq:7-2} gives
\begin{equation}\label{eq:8}
U^2\,(a - b)^2=(a+3\,b)^2-8\,b^2
\end{equation}
and by Lemma~\ref{lemma2} (in effect) this implies rational $c$ with
\begin{align}
&U\left(\frac{a}{b}-1\right)=\frac{c^2-2}{c},\notag\\
&\frac{a}{b}+3=\frac{c^2+2}{c}.     \label{eq:9-2}
\end{align}
The formula \eqref{eq:9-2} expresses $a$ as a rational function of $b$ and $c$,
and propagating this and the expression for $U$ up the preceding definitions
gives the result stated.
\end{proof}

\end{document}